\documentclass[12pt]{extarticle}
\usepackage{amsmath, amsthm,graphicx,amsfonts,amssymb,mathrsfs,showidx}
\usepackage{hyperref}
\usepackage{tikz-cd} 

\usepackage{tikz}
\usetikzlibrary{tqft}

\usepackage{relsize}
\usepackage{color}
\usepackage[english]{babel} 

\definecolor{rojo}{rgb}{1,0,0}
\usepackage[all]{xy}
\makeindex
\usepackage{hyperref}
\usepackage{color}
\definecolor{abelian}{cmyk}{0.50,0,1,.4}
\definecolor{noabelian}{cmyk}{0.94,0.54,0,0}
\definecolor{rojo}{cmyk}{0,1,1,0}
\definecolor{verde}{cmyk}{0.91,0,0.88,0.12}
\xyoption{arc}

\makeatletter
\newcommand{\customlabel}[2]{%
   \protected@write \@auxout {}{\string \newlabel {#1}{{#2}{\thepage}{#2}{#1}{}} }%
   \hypertarget{#1}{#2}
}
\makeatother

\newtheorem{thm}{Theorem}
\newtheorem{cor}[thm]{Corollary}
\newtheorem{lem}[thm]{Lemma}
\newtheorem{prop}[thm]{Proposition}
\newtheorem{example}{Example}[section]
\theoremstyle{definition}
\newtheorem{defn}[thm]{\textbf{Definition}}

\theoremstyle{definition}

\theoremstyle{remark}
\newtheorem{rem}[thm]{Remark:}
\theoremstyle{noter}

\newcommand{\op}{\operatorname}
\newcommand {\CC} {\mathbb{C}}
\newcommand{\ben}{\begin{equation}}
\newcommand{\een}{\end{equation}}
\newcommand{\bena}{\begin{equation*}}
\newcommand{\eena}{\end{equation*}}

\newcommand{\ma}{\mathcal}
\newcommand{\To}{\longrightarrow}

\newcommand{\ZZ}{\mathbb{Z}}
\newcommand{\Tot}{\longmapsto}

\def\ZZ{\mathbb{Z}}

\title{Extending free group action on surfaces}

\author{Jes\'us Emilio Dom\'inguez\thanks{Universidad Autonoma de Sinaloa, M\'exico. {\em e-mail: }{\tt jedguez@gmail.com}}
and Carlos Segovia\thanks{Instituto de Matem\'aticas UNAM-Oaxaca,  M\'{e}xico. {\em e-mail: }{\tt csegovia@matem.unam.mx}}}

\begin{document}

\maketitle

\begin{abstract}
The present work introduces new perspectives in order to extend finite group actions from surfaces to 3-manifolds. We consider the Schur multiplier associated to a finite group $G$ in terms of principal $G$-bordisms 
in dimension two, called $G$-cobordisms. We are interested in the question of when
a free action of a finite group on a closed oriented surface extends to a non-necessarily free action on a 3-manifold. 
We show the answer to this question is affirmative for abelian, dihedral, symmetric and alternating groups.
As an application of our methods, we show that every non-necessarily free action of abelian groups (under certain conditions) and dihedral groups on a closed oriented surface extends to $3$-dimensional handlebody.
\end{abstract}

\section*{Introduction}
\label{intro}

Let $\Omega_n^{SO}(G)$ be the free $G$-bordism group in dimension $n$ from Conner-Floyd \cite{CF} and denote by 
$\Omega_{n+1}^{SO,\partial free }(G)$, the $G$-bordism group of $(n+1)$-dimensional manifolds with a non necessarily free $G$-action which restricts to a free action over the boundary.   
We are interested in knowing what the image of the following map is
\ben\label{e345}
\Omega_{n+1}^{SO,\partial free }(G)\longrightarrow\Omega_n^{SO}(G)\,.
\een
For $n=2$, the group $\Omega_2^{SO}(G)$ has been studied extensively with the name of the Schur multiplier \cite{SM} (denoted by $\ma{M}(G)$).
When this group vanishes, the map \eqref{e345} is surjective, such is the case for cyclic groups, groups of deficiency zero, see \cite{SM}. 
For free actions of abelian groups and dihedral groups, the extension was given by Reni-Zimmermann \cite{RZ} and Hidalgo \cite{Hida}. Obstructions for the surjectivity of the map \eqref{e345} are constructed by Samperton \cite{ES}, considering the quotient by the homology classes represented by tori.

Our approach considers the elements of the group $\Omega_2^{SO}(G)$ represented by what we call $G$-cobordisms in dimension two.
These are diffeomorphism classes of principal $G$-bundles over (closed) surfaces \cite{AC}. 
We say that a $G$-cobordism is {\it extendable} if it has a representative given by a principal $G$-bundle over a surface $S$, which is the boundary 
of a $3$-dimensional manifold $M$ with an action of $G$.

For $G$ a finite abelian group, we show that every $G$-cobordism over a closed surface is extendable, see Theorem \ref{extAbe}. For this, we decompose any $G$-cobordism into small pieces given by $G$-cobordisms over a closed surface of genus one, 
which are extendable, as we will see in Proposition \ref{prC1}. 

For the dihedral group $D_{2n}$, we focus in the case $n=2k$ since the Schur multiplier $\ma{M}(D_{2n})$ vanishes for $n=2k+1$. 
Similar to the abelian case, we decompose every $D_{2n}$-cobordism into a finite product of the generator with base space of genus one, which is induced by a reflection and the rotation by 180 degrees, see Corollary \ref{cordie}.

For the symmetric group $S_n$, the Schur multiplier is non-trivial for $n> 3$ and in that case it is equal to $\ZZ_2$. In Proposition \ref{propsym}, we prove that there is a generator with base space of genus one, which is induced by any two disjoint transpositions. A similar argument works for the alternating group $A_n$, where for $n=6,7$, we use the Sylow theory of the Schur multiplier that is shown in Proposition \ref{Sylow}.

In summary, we have the following result.

\begin{thm}\label{thm1}
For $G$ a finite abelian group or $G\in \{D_{2n},S_n,A_n\}$, every $G$-cobordism over a closed oriented surface is extendable.
\end{thm}

We have the following applications for extending non-necessarily free actions over surfaces to 3-dimensional handlebodies:
\begin{itemize}
    \item[i)] In Theorem \ref{teore1}, the actions of abelian groups have two types of fixed points, which are the ones induced by hyperelliptic involutions and pairs of ramification points with complementary monodromies (signature $>2$).
    An unfolding process is performed by first considering the quotient by the hyperelliptic involutions and after some modifications, we reduce the problem to the extension of free actions. 
    
\item[ii)]  In Theorem \ref{teore2}, the actions of dihedral groups reduce to a finite product of an specific generator. We extend the action for this generator and for the surfaces realized by the products.
\end{itemize}
These results were proven before by different methods, by Reni-Zimmermann \cite{RZ} and Hidalgo \cite{Hida}.

This article is organized as follows. In Section \ref{sec0}, we review the concept of $G$-cobordism and the Schur multiplier, as well as the relations between them. In Section \ref{disi}, we give explicit generators for the Schur multiplier of the dihedral, the symmetric and the alternating groups. Finally, in Section \ref{nonfree} we construct the extensions of the free actions on closed oriented surfaces for the dihedral, symmetric and alternating groups. 
Additionally, for non necessarily free actions on closed oriented surfaces of abelian groups (under certain conditions) and dihedral groups, we construct the extensions given by 3-dimensional handlebodies.

{\bf Acknowledgements:}  the first author thanks the Academia Mexicana de Ciencias for the opportunity of participating in the Scientific Research Summer of 2020.
The second author is supported by c\'atedras CONACYT and Proyecto CONACYT ciencias b\'asicas 2016, No. 284621.
We would like to thank Bernardo Uribe and Eric Samperton for their helpful conversations.

\section{Preliminaries}
\label{sec0}

In this section we review in detail the definitions and properties of the theory of $G$-cobordisms introduced in \cite{AC,car}. In addition, we discuss some important facts about the Schur multiplier.

\subsection{$G$-cobordisms}
\label{sec1}

Throughout the article, $G$ denotes a finite group and $1\in G$ the neutral element. Also, we consider right actions of the group $G$, and all the surfaces are oriented.

\begin{defn}
Let $\Sigma$ and $\Sigma'$ be $d$-dimensional closed, oriented smooth ma\-nifolds. A cobordism between $\Sigma$ and $\Sigma'$ is a $(d+1)$-dimensional oriented smooth manifold $M$, with boundary diffeomorphic to $\Sigma\sqcup-\Sigma'$, where $-\Sigma'$ is $\Sigma'$ with the reverse orientation. Two cobordisms $M$ and $M'$ are equivalent if there exists a diffeomorphism $\phi:M\To M'$ such that we have the commutative diagram
\ben
\xymatrix{&M\ar[dd]^\phi&\\\Sigma\ar[ru]\ar[rd]&&\Sigma'\ar[lu]\ar[ld]\\&M'&\,.}
\een 
\end{defn}

\begin{defn}A {\it principal $G$-bundle} over a topological space $X$, consists of a fiber bundle $\pi:E\rightarrow X$ where the group $G$ acts freely and transitively over each fiber. 
\end{defn}

\begin{example}In dimension one, for every $g\in G$, we construct the principal $G$-bundle $P_g\rightarrow S^1$ obtained by attaching the ends of $[0,1]\times G$ via multiplication by $g$, i.e., $(0,h)$ is identified with $(1,gh)$ for every $h\in G$. 
This construction $P_g=[0,1]\times G/\sim_g$ projects to the circle by restriction to the first coordinate, and the action $P_g\times G\rightarrow P_g$ is defined by right multiplication on the second coordinate.
Any principal $G$-bundle over the circle is isomorphic to some $P_g$, and $P_g$ is isomorphic to $P_h$ if and only if $h$ is conjugate to $g$.  
\end{example}

Throughout the paper, we refer to the element $g\in G$ as the {\it monodromy} associated to the corresponding principal $G$-bundle $P_g$. In the case of the neutral element of the group $G$, we say that the monodromy is trivial.

\begin{defn}
Let $\xi:P\rightarrow \Sigma$ and $\xi':P'\rightarrow \Sigma'$ be principal $G$-bundles. 
A {\it $G$-cobordism} between $\xi$ and $\xi'$ is a principal $G$-bundle $\epsilon:Q\rightarrow M$, with diffeomorphisms for the boundaries $\partial M\cong S\sqcup- S'$ and $\partial Q\cong P\sqcup -P'$, which match with the projections and the restriction of the action. 
Two $G$-cobordisms $\epsilon:Q\rightarrow M$ and ${\epsilon}':Q'\rightarrow M'$ define the same class if $M$ and ${M}'$ are equivalent as cobordisms by a diffeomorphism $\phi:M\rightarrow {M}'$, $Q$ and $Q'$ are equivalent as cobordisms by a $G$-equivariant diffeomorphism $\psi:Q\To Q'$, and in addition, we have the commutative diagram  
\ben
\xymatrix{Q\ar[r]^\psi\ar[d]_\epsilon&Q'\ar[d]^{\epsilon'}\\M\ar[r]_\phi & M'\,.}
\een
\end{defn}

\begin{example}
\begin{enumerate}
\item A $G$-cobordism from $P_g$ to $P_h$ ($g,h\in G$) with base space the cylinder, is given by an element $k\in G$ such that $h=kgk^{-1}$. 

\item A $G$-cobordism with entry the disjoint union $P_g\sqcup P_h$ and exit $P_{gh}$, with base space the pair of pants, is a G-deformation retract\footnote{By a $G$-deformation retract we mean that the homotopy is by means of principal $G$-bundles.} of a principal G-bundle
over the wedge $S^1\vee S^1$.

\item There is only one $G$-cobordism over the disk and every representative is a trivial bundle.

\item Take as base space a two dimensional handlebody of genus $n$ with one boundary circle. A $G$-cobordism depends on elements $g_i,k_i\in G$, for $1\leq i\leq n$, with monodromy for the boundary circle given by the product $\prod_{i=1}^n[k_i,g_i]$. 

\end{enumerate}
\end{example}
In Figure \ref{fig1}, we have pictures for the $G$-cobordisms over the cylinder, the pair of pants and the disc.
For these pictures, we draw from left to right the direction for our cobordisms.
Also, every circle is labelled with the correspondent monodromy and for every cylinder we write inside the element of the group with which we do the conjugation.
\begin{figure}
    \centering
    \includegraphics[scale=1.2]{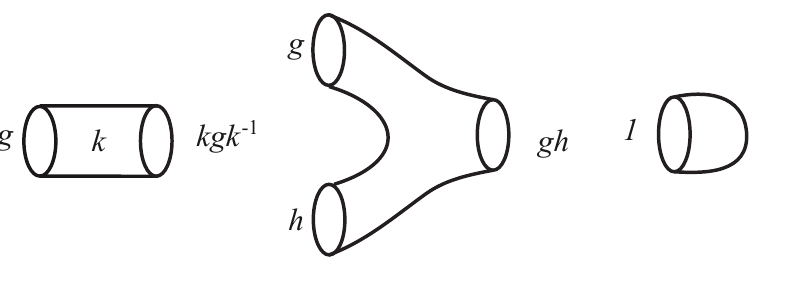}
    \caption{$G$-cobordism over the cylinder, the pair of pants and the disc.}
    \label{fig1}
\end{figure}

In the left side of Figure \ref{fig4}, we have a $G$-cobordism over a genus one handlebody. Additionally, in the right side, we represent an equivalent manner to see this $G$-cobordism.
\begin{figure}
    \centering
    \includegraphics{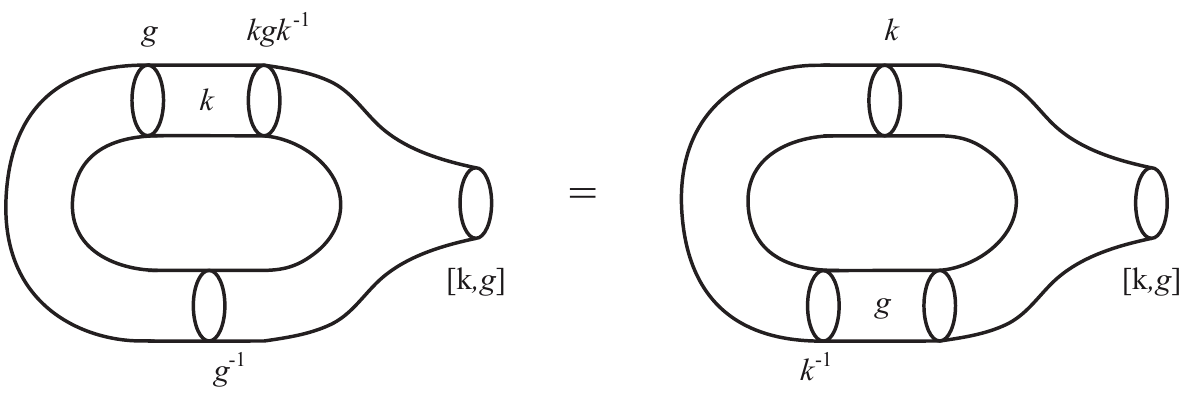}
    \caption{Two equivalent $G$-cobordisms over a handle of genus one.}
    \label{fig4}
\end{figure}

If a $G$-cobordism over a closed connected surface is cut along a simple closed se\-pa\-rating curve\footnote{A simple closed curve in a surface is separating if the cut surface is not connected.}, the monodromy of the resulting curve lies inside the commutator group, as shown in the following proposition. 

\begin{prop}\label{pro1}
For a $G$-cobordism over a closed connected surface $S$, the monodromy of every embedded simple closed separating curve in $S$ lies in the commutator group $[G,G]$.
\end{prop}

\begin{defn}
A $G$-cobordism of dimension two, over a closed surface, is {\it extendable} if for some representative principal $G$-bundle $P\rightarrow S$, with the action $\alpha:P\times G\rightarrow P$, there exits a $3$-dimensional manifold $M$ with boundary $\partial M=P$, with an action of $G$ of the form $\overline{\alpha}:M\times G\rightarrow M$, which extends $\alpha$, i.e., we have the commutative diagram 
\ben
\xymatrix{P\times G\ar@{_(->}[d]\ar[r]^\alpha & P\ar@{_(->}[d] \\  M\times G\ar[r]_{\overline{\alpha}} & M\,.}
\een
\end{defn}


\begin{prop}\label{prC1}
Any $G$-cobordism over a closed surface of genus one is extendable.
\end{prop}
\begin{proof}Consider a principal $G$-bundle representing the given $G$-cobordism over the closed surface of genus one. 
It is enough to prove the case in which the total space of the bundle is a connected space. Moreover, the action of $G$ over the total space can be modified by an isotopy resulting in an action which depends completely on a pair of monodromies $(g,k)$, which are associated to two curves in the torus that intersect once. Denote by $P_g$ and $P_k$ the two principal $G$-bundles associated to these two curves.
It follows that the action of $G$ over the total space is given by the product of the total spaces $P_g$ and $P_k$.
Because of the assumptions, at least one of $P_g$ or $P_k$ is a connected space, let us assume that it is $P_g$.
The extension of the action of $G$ is through the 3-dimensional handlebody constructed as follows. First,  
consider the disc $D$ as the union $(S^1\times (0,1])\cup \{0\}$, where $0$ is the center. For each circle $S^1\times \{r\}$, with $r\in (0,1]$, we take as monodromy the element $g$ so that the principal $G$-bundle is $P_g$, and we take the center $0$ as fixed point. Thus, over the disc we have the rotation by $2\pi/|g|$, with $|g|$ the order of $g\in G$. Taking the product of this disc together with the induced principal $G$-bundle $P_k$, we obtain the extension which makes the $G$-cobordism extendable. 
\end{proof}

\begin{rem}
We want to emphasize why the construction given in the previous proposition does not work for closed surfaces with genus $>1$.
The reason is that the set of fixed points should be a smooth submanifold, which we can not assure for genus $>1$, since we have points where three lines meet. 
\end{rem}

Now, we apply the previous results to abelian groups.

\begin{thm}\label{extAbe}
For $G$ a finite abelian group, any $G$-cobordism is extendable.
\end{thm}
\begin{proof}
Consider a $G$-cobordism over a connected closed surface.
By Proposition \ref{pro1}, we can write this $G$-cobordism as a connected sum of $G$-cobordisms with base space of genus one. 
This connected sum is done for the total space along trivial bundles over a circle. Since the connected sum is in the same bordism class as the disjoint union, we are decomposing this $G$-cobordism as a disjoint union of $G$-cobordisms with base space a closed surface of genus one. Because of Proposition \ref{prC1}, any $G$-cobordism over a closed surface of genus one is extendable, so the theorem follows.
\end{proof}

\subsection{The Schur multiplier}
\label{secschur}

The study of this theory began in 1904 by Isaai Schur in order to study the projective representations of groups. Nowadays, the Schur multiplier represents three different isomorphic groups given by the second free bordism group $\Omega_2^{SO}(G)$, the second homology group $H_2(G,\ZZ)$ and the second cohomology group $H^2(G,\CC^*)$. 

\begin{defn}
 Let $\langle G, G\rangle$ be the free group on all pairs $\langle x, y\rangle$, with $x,y\in G$. There is a natural homomorphism of $\langle G, G\rangle$ onto the commutator group $[G, G]$, which sends $\langle x, y\rangle$ into $[x, y]$. Consider the kernel $Z(G)$ of this homomorphism and the normal subgroup $B(G)$ of $\langle G, G\rangle$ generated by the relations
\begin{eqnarray}\label{four1}
		&\left<x,x\right>&\sim 1\,,\\\label{four2}
	&\left<x,y\right>&\sim \left<y,x\right>^{-1}\,,\\\label{four3}
	&\left<xy,z\right>&\sim \left<y,z\right>^x\left<x,z\right>\,,\\\label{four4}
	&\left<y,z\right>^x&\sim \left<x,[y,z]\right>\left<y,z\right>\,,
\end{eqnarray}
where $x,y,z\in G$ and $\langle y,z\rangle^x=\langle y^x,z^x\rangle = \langle xyx^{-1}, xzx^{-1}\rangle$. The Schur multiplier is defined as the quotient group 
\ben
\ma{M}(G):=\frac{Z(G)}{B(G)}\,.
\een
\end{defn}

Miller \cite{Mil} shows that the quotient $Z(G)/B(G)$ is canonically isomorphic to the Hopf's integral formula 
\ben\label{hopf}
\ma{M}(G)\cong\frac{R\cap [F,F]}{[F,R]}\,,
\een
where $G=\langle \,F\,|\,R\,\rangle$. Moreover, in \cite{Mil} there are 
some consequent relations, which we enumerate in the following theorem.
\begin{thm}[\cite{Mil}]
The following relations can be deduced from \eqref{four1}-\eqref{four4}:
 \begin{eqnarray}\label{four5}
 &\left<x,yz\right>&\sim \left<x,y\right>\left<x,z\right>^y\,,\\\label{four6}
 &\left<x,y\right>^{\left<a,b\right>}&\sim \left<x,y\right>^{[a,b]}\,,\\\label{four7}
 & \left[\left<x,y\right>,\left<a,b\right>\right]&\sim \left<[x,y],[a,b]\right>\,,\\\label{four8}
 & \left<b,b'\right>\left<a_0,b_0\right> &\sim \left<[b,b'],a_0\right>\left<a_0,[b,b']b_0\right>\left<b,b'\right>\,,\\\label{four9}
 & \left<b,b'\right>\left<a_0,b_0\right> &\sim \left<[b,b']b_0,a_0\right>\left<a_0,[b,b']\right>\left<b,b'\right>\,,\\\label{four10}
 & \left<b,b'\right>\left<a,a'\right> &\sim \left<[b,b'],[a,a']\right>\left<a,a'\right>\left<b,b'\right>\,,\\\label{four11}
 &\left<x^n,x^s\right> &\sim 1\hspace{1cm}n=0,\pm1,\cdots;s=0,\pm1,\cdots\,,
 \end{eqnarray}for $x,y,z,a,b,a',b',a_0,b_0\in G$.
\end{thm}

The connection with bordism relates the elements of $Z(G)$ by means of the assignment 
\ben\label{corrbor}
\langle x_1,y_1 \rangle \langle x_2,y_2\rangle \cdots\langle x_n,y_n \rangle \longmapsto (y_n,x_n)(y_{n-1},x_{n-1})\cdots (y_1,x_1)\,,
\een
where the sequence in the right defines the generating monodromies for a $G$-cobordism over a closed surface of genus $n$ as in Figure  \ref{fig3}. 
\begin{figure}[h!]
    \centering
    \includegraphics[scale=0.8]{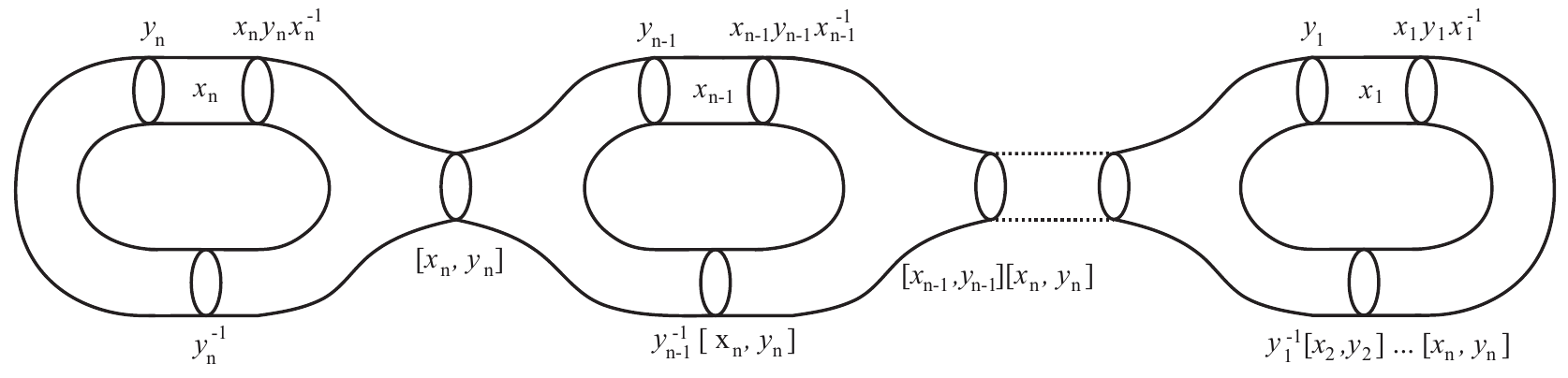}
    \caption{The $G$-cobordism associated to the sequence $(y_n,x_n)(y_{n-1},x_{n-1})\cdots (y_1,x_1)$, with $y_i,x_i\in G$ and $\prod_{i=1}^n [x_i,y_i]=1$.}
    \label{fig3}
\end{figure}

Indeed, the previous four relations \eqref{four1}, \eqref{four2}, \eqref{four3} and \eqref{four4}, are interpreted in bordism as follows:
\begin{itemize}
\item[(i)] For \eqref{four1} and \eqref{four2}, we consider the $G$-cobordism defined by the pairs $(x,x)$ and $(x,y)(y,x)$, respectively.
We represent these $G$-cobordisms in the left side of Figure \ref{fig80}, respectively. 
\begin{figure}
          \centering 
             \includegraphics[scale=0.85]{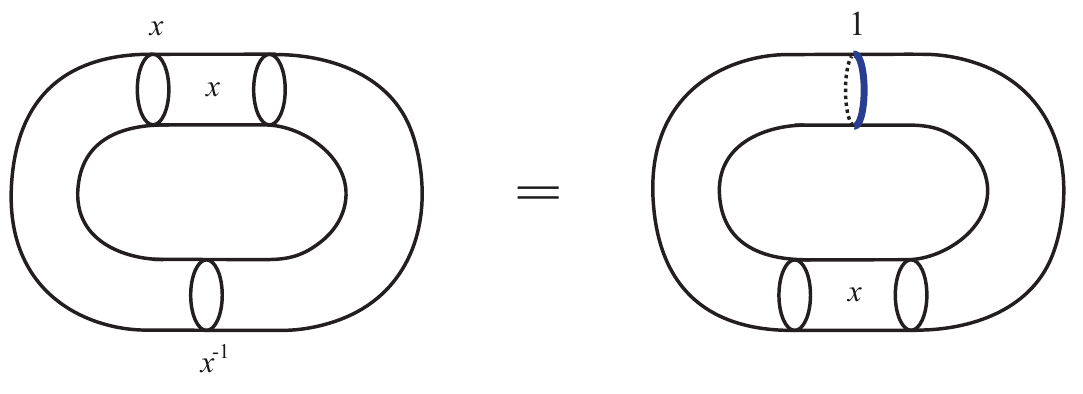}
          \includegraphics[scale=0.8]{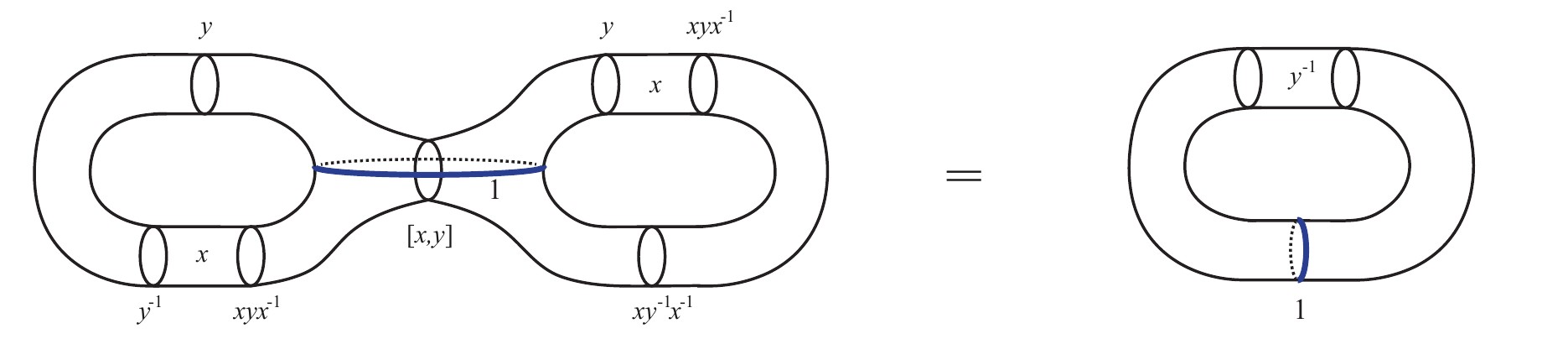}
         \caption{The $G$-cobordisms associated to the pairs $(x,x)$ and $(x,y)(y,x)$.}
         \label{fig80}
        \end{figure}
In the top of Figure \ref{fig80}, we apply the Dehn twist diffeomorphism and obtain that the conjugation becomes the neutral element $1\in G$. In the bottom of Figure \ref{fig80}, we cut along a trivial monodromy to reduce the genus by one. Notice that the $G$-cobordisms in the left side of Figure \ref{fig80} are null bordant since we can cut along a trivial monodromy eliminating the hole of the handle. 

\item[(ii)] For \eqref{four3} and \eqref{four4}, we obtain a $G$-cobordism, over a handlebody of genus two, where we can find 
a curve with trivial monodromy which reduces the genus by one. In Figure \ref{fig8}, we represent these identifications, respectively.
        \begin{figure}
          \centering 
          \includegraphics[scale=0.8]{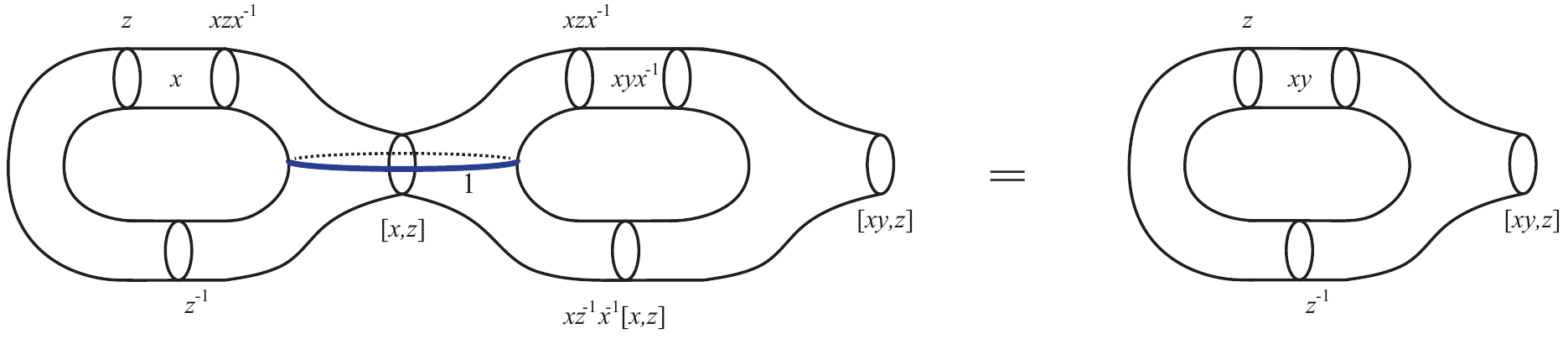}
          \includegraphics[scale=0.8]{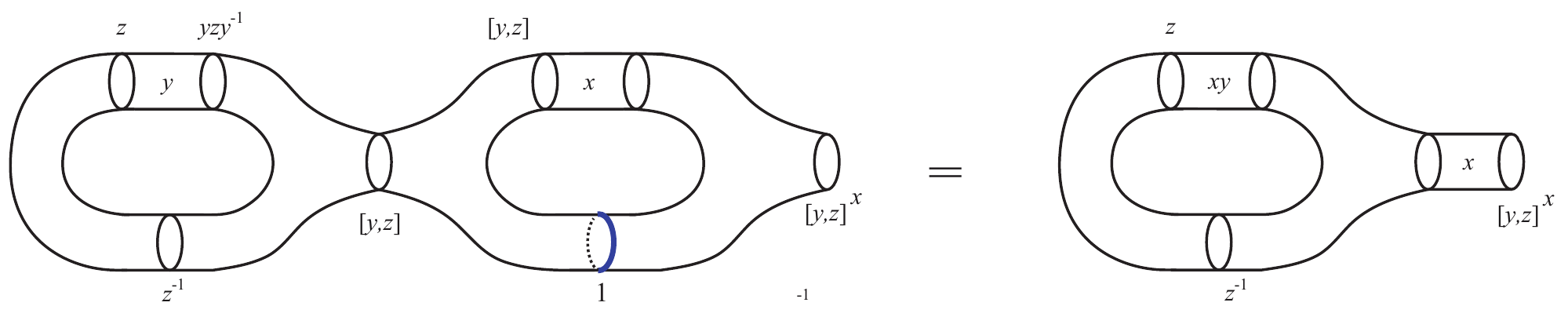}
         \caption{Reduction of genus through the cutting along a trivial monodromy.}
         \label{fig8}
        \end{figure}
\end{itemize}

Now, we focus on the Sylow theory of the Schur multiplier. We use Definition \ref{def1} and Theorem \ref{kar}, in order to show Proposition \ref{Sylow}.

\begin{defn}\label{def1}
For a subgroup $H\subset G$, there are the following induces maps:
\begin{itemize}
    \item[i)] the restriction map, denoted by $\op{res}:\ma{M}(G)\rightarrow \ma{M}(H)$, which associates to a $G$-cobordism over a closed surface, the restriction of the action to the subgroup $H$.
    \item[ii)] the corestriction map, denoted by $\op{cor}:\ma{M}(H)\rightarrow \ma{M}(G)$, which starts with a principal $H$-bundle $P\rightarrow S$ and associates the Borel construction 
    $P\times_H G$ produced by the quotient of the product $P\times G$ with the action of $H$ of the form $(x,g)h=(xh,h^{-1}g)$. The group $G$ has a free action over the  the Borel construction by $[x,g]\hat{g}=[x,g\hat{g}]$.
\end{itemize}
\end{defn}
In general, these maps extend to non-necessarily free actions, in particular, for the $G$-bordism groups $\Omega_{3}^{SO,\partial free}(G)$ of $3$-dimensional manifolds with a non necessarily free $G$-action which restricts to a free action over the boundary. 

\begin{thm}[\cite{SM}]\label{kar} Let $P$ be a Sylow $p$-subgroup of $G$ and let $\ma{M}(G)$ the $p$-component of the Schur multiplier $\ma{M}(G)$. Then the restriction map $\op{res}:\ma{M}(G)\rightarrow \ma{M}(P)$ induces an injective homomorphism $\ma{M}(G)_p\rightarrow \ma{M}(P)$,
and the corestriction map $\op{cor}:\ma{M}(P)\rightarrow \ma{M}(G)$ induces a surjective homomorphism $\ma{M}(P)\rightarrow \ma{M}(G)_p$.
\end{thm}

\begin{prop}\label{Sylow}
For a finite group $G$, and $\op{Syl}(G)$ the set of isomorphism classes of Sylow subgroups of $G$, if any element $Q$ in $\op{Syl}(G)$ satisfies that any $Q$-cobordism is extendable, then any $G$-cobordism is extendable.
\end{prop}
\begin{proof}
For $n=|G|$ and $n=p^km$ with $p\not|m$, consider an element $f\in\ma{M}(G)_p$, hence the composition \ben F:=\op{cor}\circ\op{res}:\ma{M}(G)_p\longrightarrow\ma{M}(G)_p\,,\een
is given by $f\Tot f^m$, which is an automorphism of $\ma{M}(G)_p$. By the assumptions, the restriction $\op{res}(f)$ is extendable by a $3$-manifold $M$ with an action of $Q$, therefore, applying the corestriction we obtain that $f^m$ is extendable by the $3$-manifold $\op{cor}(M)$. Similarly, we can start with $F^{-1}(f)$ and we get that $f$ is extendable and the proposition follows.
\end{proof}

\section{Generators for the Schur multiplier}
\label{disi}

In this section we give explicit generators for the Schur multiplier of the dihedral, the symmetric and the alternating groups.

\subsection{Dihedral group}
	For $n\geq 3$, the dihedral group is the group of symmetries of the $n$-regular polygon (with $D_2=\ZZ_2$, $D_{4}=\ZZ_2\times \ZZ_2$), and presentation 
\begin{equation}\label{diedral} D_{2n}=\langle a,b:a^2=1,b^2=1,(ab)^n=1\rangle\,,\end{equation}
where $c:=ab$ is the rotation of $2\pi/n$. The Schur multiplier has the form
\begin{equation}
    \ma{M}(D_{2n})=\left\{
    \begin{array}{cl}
        0 & n=2k+1\,,  \\
        \ZZ_2 & n=2k\,.
    \end{array}
    \right.
\end{equation}
In order to find a generator we show the following.

\begin{prop}\label{pr1}We obtain the following identifications:
\begin{itemize}
    \item[(i)] $\langle c^i,c^j\rangle \sim 1$,
    \item[(ii)] $\left<c^{i},ac^j\right>\sim \langle c,a\rangle^i$,
    \item[(iii)] $\left<ac^{i},c^j\right>\sim \langle c,a\rangle^{-j}$, and
    \item[(iv)] $\left<ac^i,ac^j\right>\sim \langle c,a\rangle^{j-i}$.
\end{itemize}
\end{prop}
\begin{proof}
The relation (i) follows by \eqref{four11}.  The use of \eqref{four3}, \eqref{four11} and \eqref{four5} implies 
	\begin{align}
	 &\left<c^i,ac^j\right>\sim\left<c^{i},a\right>\left<c^{i},c^{j}\right>^{a}\sim\left<c^{i},a\right>\,,\\
	 &\left<c^i,ac^j\right>=\left<cc^{i-1},ac^j\right>\sim\left<c^{i-1},ac^j\right>^c\left<c,ac^j\right>\sim\left<c^{i-1},a\right>\left<c,a\right>\,. 
		\end{align}
Therefore, we obtain the relation 
\ben
\left<c^{i},a\right>\sim\underbrace{\left<c,a\right>\left<c,a\right>\cdots\left<c,a\right>}_i\,,
\een
which implies (ii). By \eqref{four2} we obtain (iii) as follows
\ben\left<ac^{i},c^j\right>\sim \left<c^j,ac^{i}\right>^{-1}\sim \left<c,a\right>^{-j}\,.\een
Finally, we use \eqref{four5} and \eqref{four1},
\ben\left<ac^{i},ac^j\right>\sim \left<c^i,ac^{j}\right>^{a} \left<a,ac^j\right>\sim\left<c^{-1},a\right>^i \left<a,c^j\right>^a\,,\een
and by \eqref{four2} we obtain (iv).
\end{proof}

\begin{cor}\label{cordie}
    For $n=2k$, the generator of the group $\ma{M}(D_{2n})$ is represented by the element $\langle c^k,a\rangle$.
\end{cor}

\subsection{Symmetric group}

The symmetric group $S_n$ is composed of permutation of the set $[n]=\{1,\cdots, n\}$. This group is generated by the transpositions $(ij)$ with $i,j\in[n]$.
The Schur multiplier is given as follows
\begin{equation}
    \ma{M}(S_{n})=\left\{
    \begin{array}{cl}
        0 & n\leq3\,,  \\
        \ZZ_2 & n\geq4\,.
    \end{array}
    \right.
\end{equation}

\begin{lem}\label{pipiripau}
    Let $k\in[n]$, and $\langle \sigma_1,\tau_1\rangle\cdots\langle \sigma_r,\tau_r\rangle$ be a sequence with $\sigma_i,\tau_i,\in S_n$, for $i\in\{1,\cdots, r\}$. 
    There exist a positive number $0\leq s\leq r$ and the following elements:
    \begin{itemize}
        \item[(i)] $a_i,b_i\in S_n$, with $0\leq i\leq s$, such that all $a_i,b_i$ fix $k$, and 
        \item[(ii)] $c_j,d_j\in S_n$, with $0\leq j\leq r-s$, such that for each $j$, at least one of $c_j,d_j$ does not fix $k$, 
    \end{itemize}
    with the relation 
    \ben
    \langle \sigma_1,\tau_1\rangle\cdots\langle \sigma_r,\tau_r\rangle\sim \langle a_1,b_1\rangle\cdots\langle a_s,b_s\rangle\langle c_1,d_1\rangle\cdots\langle c_{r-s},c_{r-s}\rangle\,.
    \een
    Moreover, $s$ is the amount of pairs $\langle\sigma_i,\tau_i\rangle$ such that both $\sigma_i$ and $\tau_i$ fix $k$.
\end{lem}
\begin{proof}
    	It suffices to note that for pairs $\langle a,b\rangle$ and $\langle x,y\rangle$, with $a,b,x,y\in S_n$, such that $a$ and $b$ fix $k$ there is the relation
    	\begin{align*}
    	    \left<x,y\right>\left<a,b\right>&\sim\left<a,b\right>\left<b,a\right>\left<x,y\right>\left<a,b\right>\\
    	    &\sim\left<a,b\right>\left<x^{[b,a]},y^{[b,a]}\right>\,,
    	\end{align*}    
    	where we have used \eqref{four6}. An iterative application of this process, allows us to put all terms fixing $k$ to the left in the sequence.
\end{proof}

\begin{prop}\label{propsym}
Assume $n\geq 4$, and take elements $\sigma_i,\tau_i,\sigma'_j,\tau'_j\in S_n$, for $1\leq i\leq r$ and $1\leq j\leq s$, with the same commutator, i.e., 
\begin{equation}
[\sigma_1,\tau_1]\cdots[\sigma_r,\tau_r]=[ \sigma'_1,\tau'_1]\cdots[\sigma'_s,\tau'_s]\,.
\end{equation} 
Therefore, for the element $u:=\langle(1,2),(3,4)\rangle$, there is the relation  
\begin{equation}
 \langle \sigma_1,\tau_1\rangle\cdots\langle \sigma_r,\tau_r\rangle\sim u^k\langle \sigma'_1,\tau'_1\rangle\cdots\langle \sigma'_s,\tau'_s\rangle\,,
\end{equation}
with $k\in\lbrace0,1\rbrace$.
\end{prop}

\begin{proof}
First, we observe that from \eqref{four3} and \eqref{four5}, we can assume that all the elements $\sigma_i,\tau_i,\sigma'_j,\tau'_j\in S_n$ are transpositions. 
By \eqref{four4}, every pair $\langle\sigma,\tau\rangle$ with $\sigma$ and $\tau$ disjoint transpositions is in the same class as the pair $u:=\left<(1,2),(3,4)\right>$. Therefore, we can assume that the pairs are of the form $\left<(i,j),(j,k)\right>$, with $i,j$ and $k$ different numbers.

By exhaustion, the proposition follows for the symmetric group $S_n$, with $4\leq n \leq 6$. We proceed by induction for $n\geq7$ and we  suppose that for $k< n$, the generator of the Schur multiplier $\ma{M}(S_n)$ is given by the element $u:=\left<(1,2),(3,4)\right>$.
Set by $m$ the maximum of $r$ and $s$, for the sequences $\langle\sigma_1,\tau_1\rangle\cdots\langle \sigma_r,\tau_r\rangle$ and $\langle \sigma'_1,\tau'_1\rangle\cdots\langle \sigma'_s,\tau'_s\rangle$. 
For $m=1$, the proposition follows from the triviality of $\ma{M}(S_3)$.
Suppose that our proposition follows for sequences with length $l<m$. We consider the sequence 
\ben\label{week}
\langle\sigma_1,\tau_1\rangle\cdots\langle \sigma_r,\tau_r\rangle
\left(\langle \sigma'_1,\tau'_1\rangle\cdots\langle \sigma'_s,\tau'_s\rangle\right)^{-1}
\sim\langle\sigma_1,\tau_1\rangle\cdots\langle \sigma_r,\tau_r\rangle\langle \tau'_s,\sigma'_s\rangle\cdots\langle \tau'_1,\sigma'_1\rangle\,,
\een
which has trivial commutator and length given by $M:=r+s\leq 2m$. Let $x\in\lbrace1,\cdots,n\rbrace$ be the number that is fixed by the most terms of the sequence \eqref{week}. Given that the sequences have non trivial terms, each term permutes $3$ different numbers in $\{1,2,\cdots,n\}$. Therefore, the number $x$ is not fixed by at most $\frac{3(r+s)}{n}$ terms. Given that $\frac{3(r+s)}{n}\leq\frac{3(2m)}{7}<m$, hence $x$ is not fixed by at most $m-1$ terms. By Lemma \ref{pipiripau}, we can find an equivalent sequence for \eqref{week}, with the following form 
	\ben\underbrace{\left<\alpha_1,\beta_1\right>\cdots\left<\alpha_t,\beta_t\right>}_{\text{fix number x}}\underbrace{\left<\alpha_{t+1},\beta_{t+1}\right>\cdots\left<\alpha_M,\beta_M\right>}_{\text{do not fix number x}}\,,\een
	where:
	\begin{itemize}
	    \item[i)] $M-t< m$;
	    \item[ii)] the $\alpha_i,\beta_i\in S_n$, with $0\leq i\leq t$, fix $x$; and
	\item[iii)]  the $\alpha_j,\beta_j\in S_n$, with $t+1\leq j\leq M$, at least one does not fix $x$.
	\end{itemize}
	Moreover, by the proof of Lemma \ref{pipiripau}, the elements $\alpha_i,\beta_i,\alpha'_j,\beta'_j\in S_n$ are again transpositions.
Now we consider the sequences
\ben
A:=\left<\alpha_1,\beta_1\right>\cdots\left<\alpha_t,\beta_t\right>
\een
and
\ben
B:=\left(\left<\alpha_{t+1},\beta_{t+1}\right>\cdots\left<\alpha_M,\beta_M\right>\right)^{-1}=
\left<\beta_M,\alpha_M\right>\cdots\left<\beta_{t+1},\alpha_{t+1}\right>\,,
\een
where both sequences have the same commutator. Furthermore, the sequence $A$ has pairs composed by elements in $S_{n-1}$ because they fix $x$. By our induction hypothesis, for $n$, we conclude that the Schur multiplier $\ma{M}(S_{n-1})$ is generated by $u=\langle(1,2),(3,4)\rangle$. Therefore, $A\sim u^iC$ for $i\in\{0,1\}$ and $C$ is a sequence of pairs with elements in $S_{n-1}$. We can take $C$ to be of  length $< m$, as it is has the same commutator as the chain $B$ of length $<m$.
By the other induction hypothesis, for $m$, since $B$ and $C$ have length less than $m$, then there is $j\in \{0,1\}$ such that $B\sim u^j C$. This shows that the product of our initial sequences in \eqref{week} is equivalent to $u^{i-j}$ and the proof of the proposition follows. 
\end{proof}	

\begin{cor}\label{coraz2}
    For $n\geq4$, the generator of the group $\ma{M}(S_{n})$ is represented by the element $u:=\langle (1,2),(3,4)\rangle$.
\end{cor}

\subsection{Alternating group}
The alternating group $A_n$ is the normal subgroup of $S_n$ with index $2$. The Schur multiplier has the form 
\begin{equation}
    \ma{M}(A_{n})=\left\{
    \begin{array}{cl}
        0 & n\leq 3\,,  \\
        \ZZ_2 & n=4,5,  \\
        \ZZ_6 & n=6,7,   \\
        \ZZ_2 & n\geq8\,.
    \end{array}
    \right.
\end{equation}
\begin{prop}  For $n\geq4$, the element $\langle(1,2)(3,4),(1,3)(2,4)\rangle$ is nontrivial in $\ma{M}(A_{n})$. 
\end{prop}
\begin{proof}
 Because of the relations \eqref{four3} and \eqref{four5} in $\ma{M}(S_n)$, we have the following 
     \begin{align*}
         \langle(1,2)(3,4),(1,3)(2,4)\rangle&\sim\langle(3,4),(2,3)(1,4)\rangle \langle(1,2),(1,3)(2,4)\rangle\\
        &\sim\langle(3,4),(2,3)\rangle\langle(2,4),(1,4)\rangle \langle(1,2),(1,3)\rangle \langle(2,3),(2,4)\rangle
    \end{align*}
We also have from \eqref{four4}, \eqref{four5} and $\langle(2,4),(1,4)\rangle=\langle(2,3),(1,3)\rangle^{(3,4)}$ that
\begin{align*}
         \langle(2,4),(1,4)\rangle&\sim \langle(3,4),[(2,3),(1,3)]\rangle\langle(2,3),(1,3)\rangle=\langle(3,4),(1,2)(1,3)\rangle\langle(2,3),(1,3)\rangle\\
         &\sim\langle(3,4),(1,2)\rangle\langle(3,4),(2,3)\rangle\langle(2,3),(1,3)\rangle=u\langle(3,4),(2,3)\rangle\langle(2,3),(1,3)\rangle
\end{align*}
As $[(3,4),(2,3)]=[(2,3),(2,4)]$, $[(2,3),(1,3)]=[(1,3),(1,2)]$ and $\ma{M}(S_3)=0$, we have that $\langle(3,4),(2,3)\rangle\sim\langle(2,3),(2,4)\rangle$ and $\langle(2,3),(1,3)\rangle\sim\langle(1,3),(1,2)\rangle$. Therefore,
\begin{align*}
    \langle(1,2)(3,4),(1,3)(2,4)\rangle&\sim\langle(2,3),(2,4)\rangle\langle(2,4),(1,4)\rangle \langle(1,2),(1,3)\rangle \langle(2,3),(2,4)\rangle\\
    &\sim u\langle(2,3),(2,4)\rangle^2\langle(1,3),(1,2)\rangle \langle(1,2),(1,3)\rangle \langle(2,3),(2,4)\rangle\\
        &\sim u\langle(2,3),(2,4)\rangle^3 \sim u = \langle(1,2),(3,4)\rangle\,,
\end{align*}
where $\langle(2,3),(2,4)\rangle^3$ vanishes since $[(2,3),(2,4)]^3=1$ and $\ma{M}(S_3)=0$. As a consequence, the element 
$\langle(1,2)(3,4),(1,3)(2,4)\rangle$ is nontrivial in $\ma{M}(S_n)$, and hence it is also nontrivial in $\ma{M}(A_n)$ for $n\geq 4$.  
\end{proof}
\begin{cor}\label{ccor}  For $n\geq4$ and $n\not\in\lbrace6,7\rbrace$, the generator of the group $\ma{M}(A_{n})$ is represented by the element    $\langle(1,2)(3,4),(1,3)(2,4)\rangle$.
\end{cor}

\section{Extending group actions on surfaces}
\label{nonfree}
This section contains the main applications of this work. We start with free actions of abelian, dihedral, symmetric and alternating groups and then, we show that these actions extend to actions on 3-manifolds. Finally, we see the case of
non-necessarily free actions for abelian and dihedral groups. 

\subsection{Free actions on surfaces} 
\begin{defn}
Consider a compact oriented surface $S$ with a (free) group action \ben \alpha:S\times G\longrightarrow S\,.\een We say that the action is {\it extendable} if there exists a $3$-manifold $M$ with boundary $\partial M=S$, with an action of $G$ of the form 
$\overline{\alpha}:M\times G\longrightarrow M$, which extends $\alpha$, i.e., we have the commutative diagram 
\ben
\xymatrix{S\times G\ar@{_(->}[d]\ar[r]^\alpha & S\ar@{_(->}[d] \\  M\times G\ar[r]_{\overline{\alpha}} & M\,.}
\een
\end{defn}

\begin{proof}[Proof of Theorem \ref{thm1}]
By Theorem \ref{extAbe} we know that any free action of a finite abelian group is extendable. 
For dihedral groups $D_{2n}$, we have two cases to consider. One is for $n=2k+1$, but since $\ma{M}(D_{4k+2})=0$, then any free action is extendable. The other is for $n=2k$, where by Corollary \ref{cordie}, the generator of the Schur multiplier is represented by a $G$-cobordism over a closed surface of genus one, therefore, by Proposition \ref{prC1} these free actions are extendable.
Now consider free actions of the symmetric groups $S_n$, since $\ma{M}(S_n)=0$ for $n\leq 3$, it remains to prove the extension for $n\geq 4$.
Similar as for dihedral groups, by Corollary \ref{coraz2} these free actions are extendable. For the alternating groups $A_n$ the free actions are extendable for $n\leq 3$. Again, for $n\geq 4$ and $n\neq 6,7$, by Corollary \ref{ccor} these actions are extendable. In the case of free actions of $A_n$ for $n=6,7$, we notice that the Sylow subgroups of $A_6$ and $A_7$ have the following isomorphic types $\{D_8,\ZZ_3\times \ZZ_3\times \ZZ_3,\ZZ_5,\ZZ_7 \}$ and because of Proposition \ref{Sylow}, we obtain that these free actions are extendable.
\end{proof}

\subsection{Non-necessarily free action on surfaces}
Now we consider non-necessarily free actions of finite abelian groups and dihedral groups. The extension of these actions was already given by Reni-Zimmermann \cite{RZ} with $3$-dimensional methods and by Hidalgo \cite{Hida} with $2$-dimensional methods.

\begin{thm}\label{teore1}
Let $G$ be a finite abelian group with an action on a closed oriented surface where the fixed points are of two types:
\begin{itemize}
    \item[(i)]fixed points produced by hyperelliptic involutions, see Figure \ref{invo},
\begin{figure}
    \centering
    \includegraphics[scale=1.25]{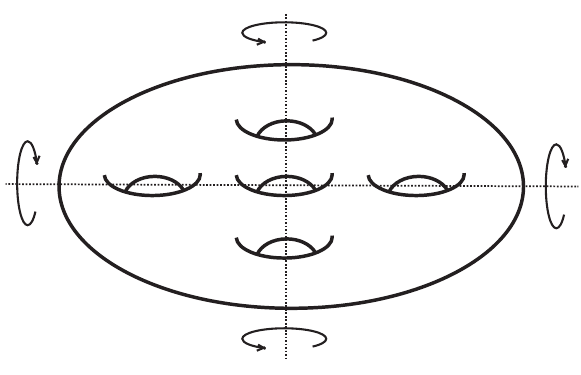}
    \caption{Hyperelliptic involutions}
    \label{invo}
\end{figure} and  
    \item[(ii)] ramification points with complementary monodromies (signature $>2$).

\end{itemize}
Then the action is extendable by a $3$-dimensional handlebody.
\end{thm}
\begin{proof}
The extension is performed in some steps. First, we consider the quotient of the surface by the hyperelliptic involutions in some order and smooth the corners with the aim to obtain a smooth closed oriented surface. 
The hyperelliptic involutions act in the set of ramification points with signature $>2$ and in the quotient we still have ramifications points grouped into pairs.
Then we connect the complementary monodromies by cylinders in order to have a free action over a closed oriented surface. From Theorem \ref{extAbe}, we have that this free action is extendable and by the proof of Proposition \ref{pr1}, this extension is by means of a 3-dimensional handlebody. 
Now we disconnect the complementary monodromies by cutting in each of the cylinders that we glued. Each cylinder has only one fixed point by the proof of Proposition \ref{prC1}. Finally, we extend the action to the original surface by an unfolding process using the hyperelliptic involutions in the reverse order in which we constructed the initial quotient surface. 
\end{proof}

\begin{thm}\label{teore2}
Every action of a dihedral group $D_{2n}$ over a closed orientable surface is extendable by a $3$-dimensional handlebody. 
\end{thm}
\begin{proof}
By Proposition \ref{pr1}, the extension problem reduces to a finite product of the same $D_{2n}$-cobordism induced by the pair $\langle c,a\rangle\sim \langle ab,a\rangle\sim \langle b,a\rangle^a$. Thus, it is enough to solve the extension problem for the pair $\langle b,a\rangle$ and for the $D_{2n}$-cobordism over the pair of pants with entries in $[D_{2n},D_{2n}]=\langle c^2\rangle$. The last reduces to the extension of $G$-cobordisms over pair of pants where $G=[D_{2n},D_{2n}]$, which follows because the group is cyclic. For the pair $\langle b,a\rangle$ we construct a representative $D_{2n}$-cobordism and there are two cases to consider:
\begin{itemize}
    \item[(i)]For $n=2k$, we consider the disjoint union of two spheres, where each one is the gluing of two $n$-gons by the boundary. 
Denote by $T$ the operation of switching from one sphere to the other and by $S$ the operation of switching from one $n$-gon to the other in the same sphere. The action of $a$ lifts to the composition $Sa=aS$ and the action of $b$ lifts to the composition $Tb=bT$. We obtain $n=2k$ fixed points over each sphere plus the north a south poles. Then for each fixed point (except the north and south pole), we remove a small disc around it and we connect the holes for opposite fixed points with a cylinder. In a similar way as for abelian groups, this action extends with the north and south pole as unique ramification points. For $k=2$, in Figure \ref{foursq}, we draw an illustrator of this construction.
\begin{figure}
    \centering
    \includegraphics[scale=2.3]{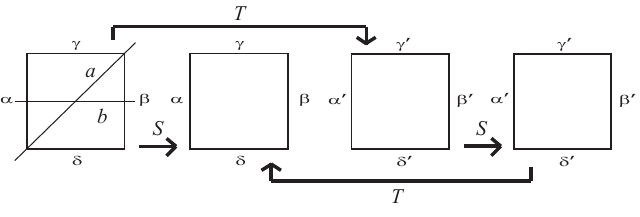}
    \caption{Representative for $\langle b,a\rangle$ with $n=2k$ ($k=1$).}
    \label{foursq}
\end{figure}
\item[(ii)]
For $n=2k+1$, we consider one sphere as the gluing of two $n$-gons by the boundary. 
Denote by $S$ the operation of switching from one $n$-gon to the other. Thus the action of $a$ lifts to the composition $Sa=aS$ and the action of $b$ lifts to the composition $Sb=bS$. We obtain $2n$ fixed points plus the north a south poles. Then we perform the same procedure to construct the extension of the action, as in the even case. 
\end{itemize}
\end{proof}

\bibliographystyle{amsalpha}
\bibliography{biblio}

\end{document}